\documentclass[12pt]{article}

\usepackage{amsmath,amssymb,mathrsfs,amsthm,xcolor, mathtools}
\usepackage[top=20truemm,bottom=20truemm,left=20truemm,right=20truemm]{geometry}

\newtheorem{Theorem}{Theorem}[section]

\newtheorem{Lemma}[Theorem]{Lemma}

\theoremstyle{definition}
\newtheorem{Definition}[Theorem]{Definition}

\newtheorem{Remark}[Theorem]{Remark}

\begin{document}
\title{Lorentzian coarea inequality} %文書タイトルの指定
\author{Hikaru Kubota\thanks{\texttt{u523299e@ecs.osaka-u.ac.jp}\textrm{, Department of Mathematics, Graduate School of Science, University of Osaka, Japan.}}} %著者名とメールアドレスの指定
  
\date{\today} %日付の指定
\maketitle %タイトルの出力

\begin{abstract}
In this article, we introduce the notion of locally uniformly $d$-controlling map between Lorentzian pre-length spaces which is preserving the diameters of causal diamonds, and through that we establish the coarea inequality for Lorentzian Hausdorff measure which is introduced by McCann and S\"{a}mann. Besides that we get a covering lemma for subsets in a Lorentzian pre-length space with a new local assumption named the local causal enlargement property, which enables us to enlarge causal diamonds. 
\end{abstract}

\section{Introduction}
Recently, low-regularity Lorentzian geometry has been studied intensively. Singularity theorems for less-regular spacetimes are investigated in $\cite{Gra;TC1}$, $\cite{1HsingC01}$, $\cite{2HsingC01}$, $\cite{Kun;Oha;Sch;Ste;HPC1}$, among others. We can find a study regarding a splitting theorem for less-regular spacetimes in $\cite{GeomC12}$. Moreover, basic causal structures and geometrical property of $C^{0}$-spacetimes are studied in $\cite{Lin;C0}$ and $\cite{Chr;Gra;C0}$ and comparison theorems for Lipschitz spacetimes are studied in $\cite{GeomC11}$. Geometry of causal structures on metric spaces or topological spaces is also investigated, which is usually called the synthetic Lorentzian geometry. In this article, we consider Lorentzian pre-length spaces. Other synthetic frameworks can be found for example in $\cite{Oct}$ and $\cite{Min;Suh1}$.

 In $\cite{Kun;Sa;pre-length}$, Kunzinger and S\"{a}mann introduced Lorentzian pre-length spaces. A Lorentzian pre-length space is a metric space $(X, d)$ equipped with causal and timelike relations, denoted by $\le$ and $\ll$, and a time separation function, denoted by $\tau$. We call the quintuplet $(X, d, \le, \ll, \tau)$ a Lorentzian pre-length space. Causally plain $C^{0}$-spacetimes hence smooth spacetimes form Lorentzian pre-length spaces. Moreover, Kunzinger and S\"{a}mann also addressed causal characters for curves on Lorentzian pre-length spaces and generalizations of causal structures of smooth spacetimes such as global hyperbolicity. Therefore we can regard Lorentzian pre-length spaces as a generalization of spacetimes. Furthermore, Lorentzian pre-length spaces can be considered as a Lorentzian counterpart of metric spaces. In $\cite{Kun;Sa;pre-length}$, Lorentzian length spaces are formulated as a notion corresponding to metric length spaces for Lorentzian pre-length spaces. Moreover, we can define curvature bounds of Lorentzian pre-length spaces in the same manner as metric spaces.

 In \cite{Kun;Sa;pre-length}, they define timelike sectional curvature bound of Lorentzian pre-length spaces by comparing geodesic triangles with those in model spaces with constant curvatures. In $\cite{Be;Ku;Oha;Ro}$ and $\cite{Be;Kun;Rot}$, the notion of timelike curvature bound via triangle comparison is revised and equivalence of notions of curvature bound on spacetimes is discussed, and in $\cite{Be;Oha;Ro;So;splliting}$ and $\cite{Be;Sa;angle}$, we can find specific structures including a splitting theorem on Lorentzian pre-length spaces under a timelike sectional curvature bound. In $\cite{Br}$, $\cite{Cav;Man;Mon;NullE2}$, $\cite{Cav;Mon}$, $\cite{McC;De}$, $\cite{B}$, and $\cite{Mon;Su}$ we can find synthetic formulations of Lorentzian pre-length spaces with Ricci curvature bounded from below through optimal transport theory, which are in the same manner as the formulation established by Sturm and Lott-Villani for metric measure spaces in \cite{LottVillani}, \cite{Strum1}, \cite{Strum2}, etc. As the notion of CD spaces includes Finsler manifolds, the class which consists of all Lorentzian pre-length spaces with timelike Ricci curvature bounded below includes Finsler spacetimes. We refer to \cite{Br;Oh} for more details.
 
 As in the case of Riemannian manifolds and metric spaces, convergence of spacetimes and Lorentzian pre-length spaces are also discussed in $\cite{Conv1}$, $\cite{Conv2}$, $\cite{Mul}$, and $\cite{Nol}$. 
 
 Therefore, we can consider the geometry of Lorentzian pre-length spaces as a generalization of the geometry of spacetimes and rebuilding of the synthetic Riemannian geometry such as the geometry of Alexandrov spaces, CAT(k)-spaces, and (R)CD spaces.

In $\cite{Mc;Sa}$, McCann and S\"{a}mann introduced a Lorentzian analog of Hausdorff measure on Lorentzian pre-length spaces. In the case of a spacetime $(M,g)$, we see that their Lorentzian Hausdorff measure on $M$ coincides with the volume measure $\mathrm{vol}^{g}$ induced by its Lorentzian metric $g$.

For a Lorentzian pre-length space $(X, d, \le, \ll, \tau)$ and $x\in X$, we denote the causal future of $x$ as $J^{+}(x)= \{p\in X\;|\; x\le p\}$ and causal past of $x$ as $J^{-}(x)$ by flipping the order of causal relation $\le$. Moreover, we denote chronological future and past as $I^{+}(x)$ and $I^{-}(x)$ in the same manner as $J^{+}(x)$ and $J^{-}(x)$ by replacing the causal relation $\le$ with the chronological relation $\ll$.

\begin{Definition}[Volume of causal diamonds]\label{Volume of diamond}
Let $N\ge0$ and $(M, d, \le, \ll, \tau)$ be a Lorentzian pre-length space. For $J(x,y)\coloneq J^{+}(x)\cap\; J^{-}(y)$ with $x\le y$ and $\tau(x,y)<\infty$, we define 
\[\rho_{N}(J(x,y))\coloneq\omega_{N}\tau(x,y)^{N},\]
where $\omega_{N}\coloneq\frac{\pi^{\frac{N-1}{2}}}{N\Gamma(\frac{N+1}{2})2^{N-1}}$ and $\Gamma(x):=\int_0^{\infty} t^{x-1}e^{-t}dt$ is the Euler's gamma function. 
\end{Definition}

\begin{Definition}[Lorentzian Hausdorff measure $\mathcal{V}^{N}$]\label{Def V^n}
Let $(X,d,\le,\ll,\tau)$ be a Lorentzian pre-length space and $N\ge0$. Set $\mathcal{J} \coloneq \{J(x,y)\;|\; x,y\in X,\;x<y\}\cup \{\emptyset\}$. By setting $\rho_{N}(\emptyset):=0$ and $\rho_{N}(J(x,y)):=\infty$ when $\tau(x,y)=\infty$, we extend $\rho_{N}: \mathcal{J} \to [0,\infty ]$. For $\delta >0$ and $A \subseteq X$, we define

\[\mathcal{V}_{d,\delta}^{N}(A):=\inf\left\{ \sum_{i=1}^{\infty} \rho_{N}(A_{i}) \;\Bigg{|}\; A_{i}\in \mathcal{J}, \mathrm{diam}_{d}(A_{i})<\delta,  A\subseteq \bigcup_{i=1}^{\infty} A_{i}\right\},\]
with $\inf\emptyset =\infty$. Then, since $\mathcal{V}^{N}_{d, \delta}(A)$ is non-increasing in $\delta>0$, we define \textit{$N$-dimensional Lorentzian Hausdorff measure} as

\[\mathcal{V}^{N}_{d}(A):=\lim_{\delta\to0} \mathcal{V}_{d,\delta}^{N}(A).\]
We call $\{A_{i}\}_{i=1}^{\infty}$ a \textit{causal $\delta$-covering} of $A$ if it satisfies $A_{i}\in \mathcal{J}$, $\mathrm{diam}_{d}(A_{i})<\delta$, and $A\subseteq \bigcup_{i=1}^{\infty} A_{i}$. When it is required to stress the specific choice of a distance function $d$, we denote the $s$-dimensional Lorentzian Hausdorff measure by $\mathcal{V}^{s}_{d}$.
\end{Definition}

In \cite{A}, we have established a comparison inequality for the Lorentzian Hausdorff measure by timelike Lipschitz maps. We continue to investigate fundamental properties of the Lorentzian Hausdorff measure and timelike Lipschitz maps.

The coarea formula plays important roles in various fields such as the geometrical measure theory and the geometry of metric spaces. The coarea inequality gives a one side inequality of the coarea formula.

\begin{Definition}\label{upper integration}
 Let $(X,\mu)$ be a measure space and $f:X\to [0, \infty]$ be a function defined in $\mu$-a.e. on $X$. Then, we define the \textit{upper integral} of $f$ as
\begin{multline*}
\int_{X}^{\ast}fd\mu\coloneq \inf\Bigg{\{}\int_{X}\phi d\mu \;\Bigg{|}\; \textrm{$\phi$ is measurable and } 0\le f(x)\le \phi(x) \textrm{ for $\mu$-a.e. $x\in X$}\Bigg{\}}.
\end{multline*}
\end{Definition}

\begin{Theorem}[Coarea inequality on metric spaces {\cite[Theorem 1.1]{Esm;Haj}}]\label{coarea inequality}
Let $(X, d_{X})$ and $(Y, d_{Y})$ be arbitrary metric spaces, $0\le t\le s<\infty$, and $E\subseteq X$. Then, for any Lipschitz map $F:X\to Y$ we have
  \[\int_{Y}^{\ast}\mathcal{H}^{s-t}_{d_{X}}(F^{-1}(y)\cap E)d\mathcal{H}^{t}_{d_{Y}}(y)\le(\mathrm{Lip}F)^{t}\frac{\omega_{s-t}\omega_{t}}{\omega_{s}} \mathcal{H}^{s}_{d_{X}}(E).\]
  Moreover, if we assume in addition $X$ is proper, $E$ is $\mathcal{H}^{s}_{d_{X}}$-measurable, and $\mathcal{H}^{s}_{d_{X}}(E)<\infty$, then the map $y\mapsto\mathcal{H}^{s-t}_{d_{X}}(F^{-1}(y)\cap E)$ is $\mathcal{H}^{t}_{d_{X}}$-measurable. Thus, it follows that for any $\mathcal{H}^{s}_{d_{X}}$-measurable function $g:X\to[0,\infty]$ we have 
   \[\int_{Y}\left\lparen\int_{F^{-1}(y)}gd\mathcal{H}^{s-t}_{d_{X}}\right\rparen d\mathcal{H}^{t}_{d_{Y}}(y)\le(\mathrm{Lip}F)^{t}\frac{\omega_{s-t}\omega_{t}}{\omega_{s}} \int_{X}gd\mathcal{H}^{s}_{d_{X}}.\]
\end{Theorem}

 In this article, we establish a Lorentzian counterpart of coarea inequality for maps satisfying properties named $\textit{uniformly $d$-controlling}$ and $\textit{timelike Lipschitz}$ and Lorentzian pre-length spaces with the $\textit{local causal enlargement property}$ and the \textit{causal estimation property}. Our strategy can be considered as an adaptation of that Esmayli and Haj\l asz established in $\cite{Esm;Haj}$ for metric spaces. In \cite{Cav;Mon;ISO}, we can find an analogous coarea type volume estimation on smooth spacetimes and Lorentzian pre-length spaces with timelike Ricci curvature bounded from below.

\begin{Definition}[Uniformly $d$-controlling and timelike Lipschitz map]\label{controlling map}
Let $(X, d_{X}, \le_{X}, \ll_{X}, \tau_{X})$ and $(Y, d_{Y}, \le_{Y}, \ll_{Y}, \tau_{Y})$ be Lorentzian pre-length spaces, and we say that a map $U:X\to Y$ is \textit{causality preserving} if $U(p)\le_{Y}U(q)$ holds for any $p,q\in X$ with $p\le_{X}q$, and \textit{timelike Lipschitz} if we can find some $\lambda>0$ such that for any $p,q\in X$ with $p\le_{X}q$, $\tau_{Y}(U(p), U(q))\le_{X} \lambda\tau_{X}(p,q)$ holds. For a timelike Lipschitz map $U$, we define 
\[\mathrm{TLip}U\coloneq\inf\{\lambda>0\;|\;\tau_{Y}(U(p), U(q))\le \lambda\tau_{X}(p,q) \textrm{ for any $p,q\in X $ with $p \hspace{-3pt}\le_{X}\hspace{-4pt} q$}\}.\]

Moreover, we say that $U:X\to Y$ is \textit{uniformly $d$-controlling} if $U$ is a causality preserving map and we can find a function $\eta:(0, \infty)\to(0, \infty)$ satisfying $\eta(s)\to 0$ as $s\to 0$ and for $J(p,q)\in \mathcal{J}$ with $\mathrm{diam}_{d_{X}}(J(p,q))<\delta$, $\mathrm{diam}_{d_{Y}}(J(U(p), U(q)))<\eta(\delta)$ holds.
\end{Definition}

\begin{Definition}[Local causal enlargement property]\label{local enlargemennt property}
For a Lorentzian pre-length space $(X, d, \le, \ll, \tau)$, $x\in X$ and $C_{1}^{x}, C_{2}^{x}\ge1$, we call a neighborhood $U_{x}$ of $x$ a \emph{causal enlargement neighborhood} for $\{C_{1}^{x}, C_{2}^{x}\}$ if any causal diamond $J(p,q)\subseteq U_{x}$ with $p\ll q$ admits a causal diamond $J(\tilde{p}, \tilde{q})$ such that
 \[\left\{ \begin{aligned} 
  & (1) \;J(\tilde{p}, \tilde{q}) \supseteq \bigcup\{J\in \mathcal{J}\;|\; J\cap J(p,q)\not=\emptyset,\; \textrm{diam}_{d}(J)\le 2\textrm{diam}_{d}(J(p,q))\},\\
  & (2) \;\mathrm{diam}_{d}(J(\tilde{p}, \tilde{q}))\le C_{1}^{x}\textrm{diam}_{d}(J(p,q)), \\
  & (3) \;\tau(\tilde{p}, \tilde{q}) \le C_{2}^{x}\tau(p, q).\\
             \end{aligned} 
             \right.\]
We call $J(\tilde{p}, \tilde{q})$ a \textit{causal enlargement} of $J(p,q)$. Moreover, we say that a Lorentzian pre-length space $(X, d, \le, \ll, \tau)$ satisfies the \textit{local causal enlargement property for $\{C_{1}^{x}, C_{2}^{x}\}$} if any $x\in X$ admits a causal enlargement neighborhood for $\{C_{1}^{x}, C_{2}^{x}\}$.

\end{Definition}

\begin{Theorem}[Lorentzian coarea inequality for upper integral]\label{LEilenberg's inequality}
  Let $(X, d_{X}, \le_{X}, \ll_{X}, \tau_{X})$ and $(Y, d_{Y}, \le_{Y}, \ll_{Y}, \tau_{Y})$ be Lorentzian pre-length spaces and $U:X\to Y$ be a uniformly $d$-controlling and timelike Lipschitz map. Let $0\le t\le s < \infty$. Assume that $Y$ satisfies the local causal enlargement property, $d_{Y}$ is a proper metric, $\mathcal{V}^{s}_{d_{Y}}(J(p,q))=0$ if $p$ and $q$ are null related, any causal diamond in $Y$ is closed, and $\mathcal{V}^{s}_{d_{Y}}$ is an outer regular measure and $\sigma$-finite on $Y$. Then, we have for any $E \subseteq X$
  \[\int_{Y}^{\ast}\mathcal{V}^{s-t}_{d_{X}}(U^{-1}(y)\cap E)d\mathcal{V}^{t}_{d_{Y}}(y)\le(\mathrm{TLip}U)^{t}\frac{\omega_{s-t}\omega_{t}}{\omega_{s}} \mathcal{V}^{s}_{d_{X}}(E).\]
\end{Theorem}

Moreover, we argue the measurability of the integrand in Theorem \ref{LEilenberg's inequality} with introducing another local structure on Lorentzian pre-length spaces, the \textit{causal estimation property}, and we get the Lorentzian coarea inequality. 
\begin{Definition} [Causal estimation property]\label{Causal estimation property}
  Let $(X, d, \le, \ll, \tau)$ be a Lorentzian pre-length space and $x\in X$. We call a neighborhood $U_{x}$ of $x\in X$ a \textit{causal estimation neighborhood} if have following two conditions:
  \begin{enumerate}
    \item[(1)]\;Any non-empty chronological diamond $I(p,q)\subseteq U_{x}$ satisfies $\textrm{diam}_{d}(I(p,q))=\textrm{diam}_{d}(J(p,q))$.
    \item[(2)]\;If $J(p,q)\subseteq U_{x}$, for every $\epsilon_{1}>1$ and $\epsilon_{2}>0$, we find a chronological diamond $I(\tilde{p}, \tilde{q})$ such that    
  \[\left\{ \begin{aligned}
&J(p,q)\subseteq I(\tilde{p}, \tilde{q}),\\
&\textrm{diam}_{d}(I(\tilde{p}, \tilde{q}))\le\epsilon_{1}\cdot\textrm{diam}_{d}(J(p,q)),\\
&\tau(\tilde{p},\tilde{q})\le\tau(p,q)+\epsilon_{2}.
 \end{aligned} \right.\]
  \end{enumerate}
 Moreover, we call such a chronological diamond $I(\tilde{p}, \tilde{q})$ a \textit{chronological} ($\epsilon_{1}$, $\epsilon_{2}$)-\textit{estimation} of $J(p,q)$. 
If any point in $X$ admits a causal estimation neighborhood, we say $X$ satisfies the \textit{causal estimation property}.
\end{Definition}
\begin{Theorem}[Lorentzian coarea inequality]\label{LORENTZIAN COAREA INEQUALITY}
  Let $(X, d_{X}, \le_{X}, \ll_{X}, \tau_{X})$ and $(Y, d_{Y}, \le_{Y}, \ll_{Y}, \tau_{Y})$ be Lorentzian pre-length spaces, $0\le t\le s<\infty$, and $U:X\to Y$. We assume the following:
  \[\left\{ \begin{aligned}
&(1)\;\textrm{$U$ is uniformly $d$-controlling, timelike Lipschitz, and continuous with respect to $d_X$ and $d_Y$}\\
&(2)\;\textrm{$Y$ satisfies the local causal enlargement property}\\
&(3)\;\textrm{Any causal diamond in $Y$ is closed}\\
&(4)\;\textrm{$d_{X}$ and $d_{Y}$ are proper}\\
&(5)\;\textrm{$\mathcal{V}^{s}_{d_{X}}$ is an inner regular measure}\\
&(6)\;\textrm{$\mathcal{V}^{t}_{d_{Y}}$ is an outer regular measure and $\sigma$-finite on $Y$}\\
&(7)\;\textrm{$\mathcal{V}^{s}_{d_{Y}}(J(p,q))=0$ if $p$ and $q$ are null related}\\
&(8)\;\textrm{$X$ satisfies the causal estimation property}.
 \end{aligned} \right.\]
 Then, for any $E\subseteq X$ such that is $\mathcal{V}^{s}_{d_{X}}$-measurable and satisfies $\mathcal{V}^{s}_{d_{X}}(E)<\infty$, we have
 \[\int_{Y}\mathcal{V}^{s-t}_{d_{X}}(U^{-1}(y)\cap E)d\mathcal{V}^{t}_{d_{Y}}(y)\le(\mathrm{TLip}U)^{t}\frac{\omega_{s-t}\omega_{t}}{\omega_{s}} \mathcal{V}^{s}_{d_{X}}(E),\]
 and thus for any $\mathcal{V}^{s}_{d_{X}}$-measurable function $g:X\to[0,\infty]$, 
 \[\int_{Y}\left\lparen\int_{U^{-1}(y)}gd\mathcal{V}^{s-t}_{d_{X}}\right\rparen d\mathcal{V}^{t}_{d_{Y}}(y)\le(\mathrm{TLip}U)^{t}\frac{\omega_{s-t}\omega_{t}}{\omega_{s}} \int_{X}gd\mathcal{V}^{s}_{d_{X}}\]
 holds.
\end{Theorem}

This article is organized as follows. In section 2, we explain fundamental definitions and properties of Lorentzian pre-length spaces and Lorentzian Hausdorff measure, and we also introduce the causal weighted integration. In section 3, we prove the Lorentzian coarea inequality for causal weighted integration, and in section 4, by showing the coincidence of upper integration and causal weighted integration under some assumptions, we prove Theorem \ref{LEilenberg's inequality}. In section 5, we introduce the causal estimation property ans discuss the measurability of the integrand in Theorem \ref{LEilenberg's inequality} and conclude Theorem \ref{LORENTZIAN COAREA INEQUALITY}. Section 6 is for outlook.

\section*{Acknowledgements}

The author would like to thank his supervisor Shin-ichi Ohta for his remarkable supports for mathematical proofs and English grammar editing. This work is supported by JST SPRING, Grant Number JPMJST2138.

\section{Preliminaries}
 \subsection{Lorentzian pre-length spaces}
In this subsection, we review fundamental definitions and properties of Lorentzian pre-length spaces based on $\cite{Kun;Sa;pre-length}$.

\begin{Definition}[Lorentzian pre-length space $\cite{Kun;Sa;pre-length}$]
  For a metric space $(X, d)$, let $\le$ be a reflexive and transitive relation on $X$, and let $\ll$ be a transitive relation on $X$. Moreover, we assume that $\ll$ is included in $\le$. Let $\tau:X\times X \to [0, \infty ]$ be a lower semi-continuous function with respect to the topology induced by $d$. Additionally, we assume that $\tau$ satisfies the reverse triangle inequality, 
\[\tau(x, z) \ge \tau(x,y) + \tau(y, z) ,\]
for all $x,y,z\in X$ with $x \le y \le z$. Moreover, suppose that $\tau(x, y)=0$ if $x \not\le y$
and $\tau(x, y)>0$ if and only if $x\ll y$. Then we call the quintuplet $(X, d, \le, \ll, \tau)$ a \emph{Lorentzian pre-length space}. We call $\le$ and $\ll$ \emph{causal} and \emph{timelike relations}, respectively, and $\tau$ a \emph{time separation function}. By $x<y$, we mean that $x \le y$ and $x \neq y$, and we say that $p$ and $q$ are \textit{null related} if $p\le q$ and $p\not\ll q$ hold.
\end{Definition}
$\linebreak$Lorentzian pre-length spaces generalize \textit{causally plain $C^{0}$-spacetimes} hence smooth spacetimes. For such a spacetime $(M^{N}, g)$, we pick an arbitrary background Riemannian distance $d_{h}$, and then with causal and chronological relations denoted by $\le_{g}$ and $\ll_{g}$ and time separation function $\tau_{g}$ induced by the Lorentzian metric $g$, $(M^{N}, d_{h}, \le_{g}, \ll_{g}, \tau_{g})$ forms a Lorentzian pre-length space. In $\cite{Ale;Gra;Kun;Sa;cone}$, we can {find other possibly non-smooth examples of Lorentzian pre-length spaces named generalized cone structure. We can see generalized cone structures generalize FLRW-spacetimes.

Next, we summarize fundamental definitions.

\begin{Definition}[Future and past]\label{Def FuturePast}
For a Lorentzian pre-length space $(X, d, \le, \ll, \tau)$ and $x\in X$, we define \emph{chronological} and \emph{causal future} of $x$ as
\begin{description}

\item{(1)} $\mathit{I}^+(x) := \{y\in X \;|\; x\ll y\}$,

\item{(2)} $\mathit{J}^+(x) := \{y\in X \;|\; x\le y\}$. 
\end{description}
By flipping of the order of causal and chronological relations, we define $\mathit{J}^{-}(x)$ and $\mathit{I}^{-}(x)$. For $p,q\in X$ we define $J(p,q)\coloneq\{x\in X\;|\;p\le x\le q\}=J^{+}(p)\cap J^{-}(q)$ and $I(p,q)\coloneq\{x\in X\;|\;p\ll x\ll q\}=I^{+}(p)\cap I^{-}(q)$. We call the prior one \textit{causal diamond} and the latter one \textit{chronological diamond}. Notice that since the time separation function $\tau$ is lower semi-continuous, any chronological diamond is open. Moreover, we can see that 
$\{\mathit{I}(x, y) \;|\; x,y\in X\}$ is a subbase of a topology on $X$. We call this topology the \textit{Alexandrov topology}.
\end{Definition}

We can define causality of curves in the following way.
\begin{Definition}[Causal and timelike curves]\label{Def CausalTimelike curve}
Let $\gamma :I \to X$ be a non-constant locally Lipschitz continuous curve with respect to $d$ on an interval $I\subseteq \mathbb{R}$. Then, we say that $\gamma$ is a \emph{future-directed causal (timelike)} curve if for all $t_1, t_2 \in I$ with $t_1 <  t_2$, we have $\gamma (t_1) \le(\ll) \; \gamma (t_2)$. We define \emph{past-directed causal (timelike)} curves by flipping the order of causal (timelike) relation. 
\end{Definition}
$\linebreak$Similarly to the case of metric spaces, we can define the length of causal curves by time separation function.
\begin{Definition}[$\tau$-length]\label{Def tau-length}
Let $\gamma :[a,b] \to X$ be a future directed causal curve. Then, we define the $\tau$-\emph{length} of $\gamma$, denoted by $L_\tau(\gamma)$, as
\[L_\tau(\gamma) := \inf \Bigg{\{} \sum_{i = 0}^{L - 1} {\tau(\gamma (t_i), \gamma (t_{i+1}))}
\;\Bigg{|}\; L\in \mathbb{N}, a = t_0 \le t_1 \le \cdots \le t_L = b \Bigg{\}}.\]
\end{Definition} 
$\linebreak$We can see that $\tau$-length satisfies
\[L_{\tau}(\gamma|_{[x, z]}) = L_{\tau}(\gamma|_{[x, y]}) + L_{\tau}(\gamma|_{[y, z]}) \quad\mathrm{for \;all} \;x\le y\le z \in [a,b],\]
and $\tau$-length is reparametrization invariant. Moreover, we say that a causal curve $\gamma:[a,b]\to X$ is a \emph{maximal curve} if $L_{\tau}(\gamma)=\tau(\gamma(a),\gamma(b))$. For a maximal curve $\gamma:[a,b]\to X$, its $\tau$-length is larger than or equal to that of any other causal curve from $\gamma(a)$ to $\gamma(b)$. In $\cite{Kun;Sa;pre-length}$, we can find discussions regarding to the difference between the definition of causal characters of curves in Definition $\ref{Def CausalTimelike curve}$ and that in the usual way for spacetimes, i.e. for a smooth spacetime $(M,g)$ with the signature $(-, +, +,\dots, +)$, a curve $\gamma:[a,b]\to M$ called $\textit{causal (timelike)}$ if $g(\dot{\gamma}, \dot{\gamma}) \le(<) 0$ on every point in $[a,b]$.\\
We can define causal structures of Lorentzian pre-length spaces.

\begin{Definition}[Causality conditions]\label{Def causality condition}
A Lorentzian pre-length space $(X, d, \le, \ll, \tau)$ is called 
\begin{description}
\item{(1)} \emph{chronological} if $\ll$ is irreflexive, i.e. $x \not\ll x$ $\mathrm{for \;all}$ $x\in X$,

\item{(2)} \emph{causal} if $\le$ is a partial order, i.e. $x\le y$ and $y\le x$ imply $x = y$ $\mathrm{for \;all}$ $x, y\in X$,

\item{(3)} \emph{strongly causal} if the Alexandrov topology agrees with the topology induced by $d$,

\item{(4)} \emph{globally hyperbolic} if the following two conditions hold:
\begin{description}
\item{(i)} For any compact set $K\subseteq X$, there exists $C>0$ such that $L_{d}(\gamma)\le C$ for any causal curve $\gamma$ in $K$ (\textit{non-totally imprisoning}),

\item{(ii)} $J(x,y)$ is a compact set $\mathrm{for \;all}$ $x, y\in X$.

\end{description}
\end{description}

\end{Definition}
$\linebreak$
Moreover, Kunzinger and S\"{a}mann introduced local structures of Lorentzian pre-length spaces.
\begin{Definition}[Localizable Lorentzian pre-length space]\label{Localizable}
  Let $(X, d, \le, \ll, \tau)$ be a Lorentzian pre-length space. We call $X$ $\textit{localizable}$ if every $x\in X$ possesses a neighborhood $\Omega_{x}\subseteq X$ which satisfies the following properties:
  \begin{description}
    \item (i) There is a constant $C>0$ such that $L^{d}(\gamma)<C$ for all causal curve $\gamma\subseteq\Omega_{x}$, where $L^{d}(\gamma)$ is a length with regard to the distance $d$. 
    \item (ii) There is a continuous map $\omega_{x}:\Omega_{x}\times\Omega_{x}\to[0, \infty]$ such that $(\Omega_{x}, d|_{\Omega_{x}\times\Omega_{x}}, \le\hspace{-4pt}|_{\Omega_{x}\times\Omega_{x}}, \ll\hspace{-4pt}|_{\Omega_{x}\times\Omega_{x}}, \omega_{x})$ forms a Lorentzian pre-length space with the following property: For every $y\in\Omega_{x}$ $I^{\pm}(y)\cap\Omega_{x}\not=\emptyset$.
    \item (iii) For all $p,q\in \Omega_{x}$ with $p<q$, there is a future directed causal curve $\gamma_{p,q}$ from $p$ to $q$ which satisfies
   \[\left\{ \begin{aligned} & (1) \textrm{ For any causal curve }\lambda\subseteq\Omega_{x} \textrm{ from  $p$ to $q$ we have } L_{\tau}(\gamma_{p,q})\ge L_{\tau}(\lambda), \\
                             & (2) \;L_{\tau}(\gamma_{p,q})=\omega_{x}(p,q)\le\tau(p,q). \\
             \end{aligned} 
             \right.
             \]
  \end{description}
  We call such $\Omega_{x}$ $\textit{localizable neighborhood of x}$. Moreover if $\Omega_{x}$ can be chosen such that 
  \begin{description}
    \item (iv) For any $p,q\in\Omega_{x}$ timelike related, $\gamma_{p,q}$ has strictly larger $L_{\tau}$ length than any future directed causal curve from $p$ to $q$ containing a null segment, 
  \end{description}
  $X$ is called $\textit{regular localizable}$.
\end{Definition}

\begin{Remark}
Through the notion of localizability of Lorentzian pre-length space, we can define Lorentzian length space, which is a similar notion to metric length space. In Lorentzian length spaces, causal structures mentioned in Definition \ref{Def causality condition} are related as follows: $(4)\Rightarrow(3)\Rightarrow(2)\Rightarrow(1)$ (\cite[Theorem 3.26]{Kun;Sa;pre-length}).  Moreover, any globally hyperbolic Lorentzian length space satisfies that the time separation function is finite and continuous and that any two causally related points are connected by a maximal curve (\cite[Theorem 3.28, Theorem 3.30]{Kun;Sa;pre-length}).
\end{Remark}

Finally we recall Definition 1.6 for the local causal enlargement property.

 \subsection{Lorentzian Hausdorff measure and causal weighted integration}
 In $\cite{Mc;Sa}$, McCann and S\"{a}mann introduced the \emph{Lorentzian Hausdorff measure} as a notion of Lorentzian analog of Hausdorff measure. Recall Definitions \ref{Volume of diamond} and \ref{Def V^n} for the definitions of $\rho_N$ and the Lorentzian Hausdorff measure $\mathcal{V}_d^N$.

\begin{Remark}
We can see that for an integer $N\ge2$ and $x,y\in X$ with $\tau(x,y)<\infty$, 
\[\rho_{N}(J(x,y))=\mathrm{vol}^{\mathbb{R}_{1}^{N}}(\tilde{J}(\tilde{x},\tilde{y})),\]
where $\tilde{x}$ and $\tilde{y}$ are points in the $N$-dimensional Minkowski space $\mathbb{R}_{1}^{N}$ such that $\tilde{\tau}(\tilde{x},\tilde{y})=\tau(x,y)$. Recall that the above equation holds independently from the choice of $\tilde{x}$ and $\tilde{y}$.

\end{Remark}

$\linebreak$
We can easily see that $\mathcal{V}^{s}_{d}$ is a Borel measure on $X$ similarly to the case of Hausdorff measures on metric spaces. Moreover, we find various properties of Lorentzian Hausdorff measure in $\cite{Mc;Sa}$.

\begin{Theorem}[{\cite[Proposition 3.8]{Mc;Sa}}]
  Let $(X, d, \le, \ll, \tau)$ be a strongly causal Lorentzian pre-length space. Let $\gamma:[a,b]\to X$ be a future directed causal curve. Then, we have
  \[\mathcal{V}^{1}_{d}(\gamma([a,b]))\le L_{\tau}(\gamma).\]
  Moreover, if we assume that any causal diamond in $X$ is closed, the equality is attained, i.e.
  \[\mathcal{V}^{1}_{d}(\gamma([a,b]))= L_{\tau}(\gamma).\]
\end{Theorem}

\begin{Theorem}[{\cite[Theorem 4.8]{Mc;Sa}}]\label{Th Volg=V^n}
Let $(M,g)$ be a strongly causal, continuous and causally plain spacetime of dimension $N$ and $\mathrm{vol}^{g}$ be the volume measure of the Lorentzian metric $g$ on $M$. Then, $\mathrm{vol}^{g}=\mathcal{V}^{N}_{d_{h}}$ holds. Here, $d_{h}$ is an arbitrarily taken Riemannian distance of any complete Riemannian metric $h$ on $M$.

\end{Theorem}

Next, we introduce the $\textit{causal weighted integration}$ besides the upper integration in Definition 1.3. The following lemmas regarding to the upper integral follow immediately.

\begin{Lemma}\label{zero integration imply zero a.e.}
  Let $(X, \mu)$ be a measure space and $f:X\to[0, \infty]$ with zero upper integral. Then, $f=0$ $\mu$-almost everywhere.
\end{Lemma}

\begin{Lemma}\label{flipping limit}
  Let $(X, \mu)$ be a measure space and $f_{n}:X\to [0, \infty]$ be a monotone sequence of functions, i.e. $0\le f_{1}(x)\le f_{2}(x)\dots$ for $\mu$-a.e. $x\in X$. Then, set a function $f$ defined $\mu$-a.e. on $X$ as $f(x)\coloneq\displaystyle\lim_{n\to\infty}f_{n}(x)$, and we have
  \[\displaystyle\lim_{n\to\infty}\int_{X}^{\ast}f_{n}d\mu=\int_{X}^{\ast}fd\mu.\]
\end{Lemma}

\begin{Lemma}\label{estimation by step functions}
  Let $(X, \mu)$ be a measure space and $f:X\to[0, \infty]$ be a any function. We call $\phi:X\to [0, \infty]$ a step function if we find $a_{i}\in(0, \infty]$ and measurable $A_{i}\subseteq X$ such that $\phi=\displaystyle\sum_{i=1}^{\infty}a_{i}\chi_{A_{i}}$. Then, we have
  \[\int_{X}^{\ast}fd\mu=\inf\left\{\int_{X}\phi d \mu\;\Bigg{|}\;\text{$\phi$ is a step function, }0\le f(x)\le \phi(x) \text{ $\mu$-a.e. } x\in X\right\}.\]
\end{Lemma}

\begin{Definition} [Causal weighted integration]\label{Causal weighted integral}
Let $(X, d, \le, \ll, \tau)$ be a Lorentzian pre-length space, $s\ge0$, and $f:X\to [0, \infty]$. For $\delta>0$, we define a \textit{causal weighted $\delta$-covering $\{(a_{i}, J_{i})\}_{i=1}^{\infty}$} of $f$ as a countably infinite pairs of a $a_{i}\in(0, \infty]$ and a causal diamond $J_{i}\in \mathcal{J}$ such that
\[\left\{ \begin{aligned} 
  & (1) \;\textrm{diam}_{d}(J_{i})\le \delta,\\
  & (2) \;f\le \sum_{i=1}^{\infty}a_{i}\chi_{J_{i}}.
             \end{aligned} 
             \right.\]
Moreover, we set
\[\int_{X}^{\bullet}fd\mathcal{V}^{s}_{d, \delta}\coloneq\inf\left\{\displaystyle\sum_{i=1}^{\infty}a_{i}\rho_{s}(J_{i})\;\Bigg{|}\;\{(a_{i}, J_{i})\}_{i=1}^{\infty} \textrm{ is a causal weighted $\delta$-covering of $f$}\right\},\]
with $\inf\emptyset=\infty$. Notice that the above formulation is non-increasing in $\delta$, and we define the \textit{causal weighted integral of $f$} for $\mathcal{V}^{s}_{d}$ as
\[\int_{X}^{\bullet}fd\mathcal{V}^{s}_{d}\coloneq \displaystyle\lim_{\delta\to 0}\int_{X}^{\bullet}fd\mathcal{V}^{s}_{d,\delta}.\]
\end{Definition}

Here, notice that for non-negative functions, $f$ and $g$ on $X$ with $f\le g$, causal weighted integration of $g$ is larger than that of $f$.

\section{Lorentzian coarea inequality for weighted causal integration}

In this section, we prove Lorentzian coarea inequality with regard to the causal weighted integration. 
\begin{Theorem}\label{Lcoarea inequality for causal weighted integration} 
  Let $X$ and $Y$ be Lorentzian pre-length spaces, $U:X\to Y$ be a uniformly $d$-controlling and timelike Lipschitz map, $0\le t\le s <\infty$, and $E\subseteq X$. Then, we have
  \[\displaystyle\lim_{\delta\to0} \int_{Y}^{\bullet}\mathcal{V}_{d_{X},\delta}^{s-t}(U^{-1}(y)\cap E)d\mathcal{V}^{t}_{d_{Y}}\le (\mathrm{TLip}U)^{t}\frac{\omega_{s-t}\omega_{t}}{\omega_{s}} \mathcal{V}^{s}_{d_{X}}(E).\]
\end{Theorem}

Recall Definition 1.5 for the $d$-controlling and timelike Lipschitz properties.

\begin{Remark}
Notice that for a causality preserving map $U:X\to Y$, $U(J(p,q))\subseteq J(U(p), U(q))$ holds obviously though the reverse inclusion does not in general. Thus, the uniform continuity of $U$ does not imply that $U$ is uniformly $d$-controlling.
\end{Remark}

\begin{Remark}
We can find discussions of uniformly $d$-controlling maps in \cite{A} by the author though the notion of uniformly $d$-controlling map is introduced in this paper. In \cite{A}, conditions on causality preserving maps which suggest uniformly $d$-controlling property are addressed. If a causality preserving map $F$ between Lorentzian pre-length spaces $X$ and $Y$ is surjective, uniformly continuous, and preserves causality dually, i.e. $p\hspace{-3pt}\le_{X} \hspace{-4pt}q$ $\Leftrightarrow$ $F(p)\hspace{-3pt}\le_{Y}\hspace{-4pt} F(q)$, then $F$ is uniformly $d$-controlling. When the distance function on $Y$ is a null distance, we can weaken the condition, namely if a causality preserving map $F$ is uniformly continuous, it is uniformly $d$-controlling. Moreover, when the target space $Y$ is a generalized cone, uniformly continuous causality preserving map $F$ is uniformly $d$-controlling.
\end{Remark}

To make the proof clearer, we introduce the following notions.

\begin{Definition}
  Let $(X, d_{X}, \le_{X}, \ll_{X}, \tau_{X})$ and $(Y, d_{Y}, \le_{Y}, \ll_{Y}, \tau_{Y})$ be Lorentzian pre-length spaces, $E\subseteq X$, and $U:X\to Y$ be a causality preserving
map. Let $s, t\in[0,\infty)$ and $\delta\in(0,\infty]$. Then, we define
\begin{multline*}
  \Phi^{s,t}_{\delta}(U,E)\coloneq\inf\Bigg{\{}\sum_{i=1}^{\infty}\rho_{s}(J(U(a_{i}), U(b_{i})))\rho_{t}(J(a_{i}, b_{i}))\;\\\Bigg{|}\;\{J(a_{i}, b_{i})\}_{i=1}^{\infty}\textrm{ is a causal $\delta$-covering of $E$}\Bigg{\}},
  \end{multline*}
with $\inf\emptyset =\infty$. Notice that $\Phi^{s,t}_{\delta}(U,E)$ is non-increasing in $\delta$, and then we define
\[\Phi^{s,t}(U,E)\coloneq\displaystyle\lim_{\delta\to0}\Phi^{s, t}_{\delta}(U,E).\]
\end{Definition}
Then, the lemmas below conclude Theorem \ref{Lcoarea inequality for causal weighted integration}.
\begin{Lemma}\label{step1 of A}
  Let $(X, d_{X}, \le_{X}, \ll_{X}, \tau_{X})$ and $(Y, d_{Y}, \le_{Y}, \ll_{Y}, \tau_{Y})$ be Lorentzian pre-length spaces, $E\subseteq X$,  and $U:X\to Y$ be a causality preserving and timelike Lipschitz map. Let $s,t\in [0, \infty)$. Then, we have
  \[\Phi^{s,t}(U,E)\le (\mathrm{TLip}U)^{s}\frac{\omega_{s}\omega_{t}}{\omega_{s+t}} \mathcal{V}^{s+t}_{d_{X}}(E).\]
\end{Lemma}

\begin{proof}
  We can assume $\mathcal{V}^{s+t}_{d_{X}}(E)<\infty$. Take $\delta>0$ arbitrarily. Let $\{J(a_{i}, b_{i})\}_{i=1}^{\infty}$ be a causal $\delta$-covering of $E$. Then, we have
  \[\begin{aligned}
  \Phi^{s,t}_{\delta}(U, E)&\le \sum_{i=1}^{\infty}\rho_{s}(J(U(a_{i}), U(b_{i})))\rho_{t}(J(a_{i}, b_{i}))\\
                           &\le \sum_{i=1}^{\infty}\omega_{s}(\mathrm{TLip}U)^{s}\tau_{X}(a_{i}, b_{i})^{s}\omega_{t}\tau_{X}(a_{i}, b_{i})^{t}\\
                           &= (\mathrm{TLip}U)^{s}\frac{\omega_{s}\omega_{t}}{\omega_{s+t}}\sum_{i=1}^{\infty}\rho_{s+t}(J(a_{i}, b_{i})).
  \end{aligned}\]
The assumption that $U$ is causality preserving and timelike Lipschitz is required to get the inequality in the second line above. Taking the infimum over all causal $\delta$-coverings of $E$, we get
  \[\Phi^{s,t}_{\delta}(U, E)\le (\mathrm{TLip}U)^{s}\frac{\omega_{s}\omega_{t}}{\omega_{s+t}}\mathcal{V}^{s+t}_{d_{X},\delta}(E).\]
Letting $\delta\to0$, we have the desired inequality.
\end{proof}

\begin{Lemma}\label{step2 of A}
  Let $X$ and $Y$ be Lorentzian pre-length spaces, and $U:X\to Y$ be a uniformly $d$-controlling map. Let $0\le t\le s <\infty$ and $E\subseteq X$. Then, we have
  \[\int_{Y}^{\bullet}\mathcal{V}_{d_{X},\delta_{0}}^{s-t}(U^{-1}(y)\cap E)d\mathcal{V}^{t}_{d_{Y}}(y)\le \Phi^{t, s-t}(U, E)\]
  for any $\delta_{0}>0$.
\end{Lemma}

\begin{proof}
  We can assume $\Phi^{t, s-t}(U, E)<\infty$. Let $\delta_{0}>0$. Let $0<\delta\le\delta_{0}$ and $\epsilon>0$, take a causal $\delta$-covering of $E$, $\{J(a_{i}, b_{i})\}_{i=1}^{\infty}$ such that 
\[\sum_{i=1}^{\infty}\rho_{t}(J(U(a_{i}), U(b_{i})))\rho_{s-t}(J(a_{i}, b_{i})) < \Phi_{\delta}^{t, s-t}(U, E)+\epsilon.\]
Then, we also have 
\[\mathcal{V}_{d_{X},\delta_{0}}^{s-t}(U^{-1}(y)\cap E)\le \sum_{i=1}^{\infty}\rho_{s-t}(J(a_{i}, b_{i}))\cdot\chi_{J(U(a_{i}), U(b_{i}))}(y).\]
Indeed, we see $U^{-1}(y)\cap E\subseteq \bigcup\{J(a_{i}, b_{i})\;|\;y\in J(U(a_{i}), U(b_{i}))\}$. Moreover, since $U$ is uniformly $d$-controlling, we can get a function $\eta:(0, \infty)\to (0, \infty)$ such that
\[\left\{ \begin{aligned} & (1) \;\mathrm{diam_{d_{X}}}(J(a, b))<s\Rightarrow \mathrm{diam_{d_{Y}}}(J(U(a), U(b)))<\eta(s), \\
                          & (2) \;\eta(s)\to 0 \text{\;as\;} s\to 0\;. \\
             \end{aligned} 
             \right.
             \]
Notice that $\{\rho_{s-t}(J(a_{i}, b_{i})), J(U(a_{i}), U(b_{i}))\}$ forms a weighted causal $\eta(\delta)$-covering of the map, $y \mapsto \mathcal{V}_{d_{X},\delta_{0}}^{s-t}(U^{-1}(y)\cap E)$. Therefore, we have
\[\begin{aligned}
  \int_{Y}^{\bullet}\mathcal{V}_{d_{X},\delta_{0}}^{s-t}(U^{-1}(y)\cap E)d\mathcal{V}^{t}_{\eta(\delta)}   &\le\sum_{i=1}^{\infty} \rho_{s-t}(J(a_{i}, b_{i}))\rho_{t}(J(U(a_{i}), U(b_{i})))\\
  &<\Phi_{\delta}^{t, s-t}(U, E)+\epsilon.\\
  \end{aligned}\]
  Letting $\epsilon\to0$ first and $\delta\to0$ next, we can get the desired inequality.
\end{proof}
$\linebreak$
As a direct consequence of Lemmas $\ref{step1 of A}$ and $\ref{step2 of A}$, we get the desired conclusion, Theorem \ref{Lcoarea inequality for causal weighted integration}.

\section{Coincidence of the upper integration and the causal weighted integration}
In this section, we prove the coincidence of the upper integration and the causal weighted integration on Lorentzian pre-length spaces. 
\begin{Theorem}\label{equality of integration}
Let $(X, d, \le, \ll, \tau)$ be a Lorentzian pre-length space. Assume that $X$ satisfies the local causal enlargement property, $d$ is a proper metric, $\mathcal{V}^{s}_{d}(J(p,q))=0$ if $p$ and $q$ are null related, and any causal diamond in $X$ is closed. Let $s\ge0$ and assume that $\mathcal{V}^{s}_{d}$ is an outer regular measure and $\sigma$-finite on $X$. Then, for any $f:X\to [0, \infty]$, we have
\[\int_{X}^{\ast}fd\mathcal{V}^{s}_{d}=\int_{X}^{\bullet}fd\mathcal{V}^{s}_{d}.\]
\end{Theorem}

To prove Theorem \ref{equality of integration}, we provide two lemmas. The prior one is a covering lemma regarding to causal diamonds, and the latter one is about estimation of the volume of causal diamond over compact sets. Moreover, we see that combining Theorem \ref{Lcoarea inequality for causal weighted integration} and Theorem \ref{equality of integration}, Theorem \ref{LEilenberg's inequality} holds.

\begin{Lemma}\label{causal covering lemma}
  Let $(X, d, \le, \ll, \tau)$ be a Lorentzian pre-length space satisfying the local causal enlargement property. Let $E\subseteq X$, and assume there exists a compact subset $A\subseteq X$ with $E\subseteq A$ and $\delta\coloneq d(E, A^{c})>0$. Then, we can find positive numbers $C>0$, $C_{1}\ge1$, and $C_{2}\ge1$, and we have the following:
  
Let $\kappa$ be any family of causal diamonds in $\mathcal{J}$ such that
  \[\left\{ \begin{aligned} 
  & (1) \;\textit{Any $J(p,q)$ in $\kappa$ satisfies $p\ll q$,} \\
  & (2) \;E\subseteq\displaystyle\bigcup_{J(p,q)\in\kappa}J(p,q), \\
  & (3) \;E\cap J(p,q)\not= \emptyset \mathrm{  \;\;for \;any \;}J(p,q)\in\kappa, \\
  & (4) \;\sup\{\mathrm{diam}_{d}(J(p,q))\;|\;J(p,q)\in \kappa\}<C. \\
             \end{aligned}
             \right.\tag{A}\\\]
Then any element $J(p,q)\in \kappa$ admits a causal enlargement $\hat{J}\coloneq J(\hat{p},\hat{q})$ such that $\mathrm{diam}_{d}J(\hat{p}, \hat{q})\le C_{1}\mathrm{diam}_{d}(J(p,q))$ and $\tau(\hat{p},\hat{q})\le C_{2}\tau(p,q)$. Moreover, we can get a subfamily $\kappa'\subseteq\kappa$ that is pairwise disjoint and satisfies
\[E\subseteq\displaystyle\bigcup_{J'\in\kappa'}\hat{J'}, \]
independently from the choice of such enlargement $\hat{J'}$ for each $J'\in\kappa'$. If in addition we assume that any causal diamond in $X$ is closed and $\kappa$ is a fine covering of $E$, for any finite subfamily $K\subseteq \kappa'$ we have
\[E\subseteq \displaystyle\bigcup_{J\in K}J\cup\displaystyle\bigcup_{J'\in\kappa'\setminus K}\hat{J'}.\]
\end{Lemma}

\begin{proof}
    First we show that there are constants $C>0$, $C_{1}\ge1$, and $C_{2}\ge1$ such that any $J\in\mathcal{J}$ satisfying $J\cap E\not=\emptyset$ and $\mathrm{diam}_{d}(J)<C$ admits a causal enlargement $\hat{J}$ which holds $\mathrm{diam}_{d}J(\hat{p}, \hat{q})\le C_{1}\mathrm{diam}_{d}(J(p,q))$ and $\tau(\hat{p},\hat{q})\le C_{2}\tau(p,q)$. Since we assume there exists a compact subset $A\subseteq X$ with $\delta= d(E, A^{c})>0$ namely $E \subseteq A$, assigning a causal enlargement neighborhood $U_{x}$ for each $x\in A$, we can get a finite subcovering $\{U_{x_{i}}\}_{i=1}^{M}$. Then, taking a Lebesgue number $C'>0$ for the covering $\{U_{x_{i}}\}_{i=1}^{M}$, we see that for any causal diamond $J\subseteq A$ with $\mathrm{diam}_{d}(J)<C'$ there exists a causal enlargement neighborhood $U_{x_{i}}$ of a point $x_{i}$ which includes $J$. Therefore, again since $\delta>0$, for $C\coloneq \min{(C', \frac{1}{2}\delta)}$, $C_{1}\coloneq\max\{C_{1}^{x_{i}}\}_{i=1}^{M}$, and $C_{2}\coloneq\max\{C_{2}^{x_{i}}\}_{i=1}^{M}$, any $J\subseteq A$ with $\mathrm{diam}_{d}(J)<C$ admits a causal enlargement $\hat{J}$ such that $\mathrm{diam}_{d}J(\hat{p}, \hat{q})\le C_{1}\mathrm{diam}_{d}(J(p,q))$ and $\tau(\hat{p},\hat{q})\le C_{2}\tau(p,q)$. 
  
  Let $\kappa$ be a covering of $E$ which consists of causal diamonds in $\mathcal{J}$ and satisfies (A) for $C>0$. Assign a causal enlargement $\hat{J}$ for each $J\in\kappa$. Let $R\coloneq \displaystyle\sup_{J\in\kappa}\{\mathrm{diam}_{d}(J)\}<C$, and for $j\in\mathbb{N}$ we set 
  \[\kappa_{j}\coloneq\left\lbrace J\in\kappa\;\left\vert\;\frac{R}{2^{j}}<\mathrm{diam}_{d}(J)\le \frac{R}{2^{j-1}}\right.\right\rbrace.\]
  Take a pairwise disjoint subfamily $\tilde{\kappa}_{1}\subseteq\kappa_{1}$ which is also maximal with respect to inclusion. Then, we take a maximal pairwise disjoint subfamily $\tilde{\kappa}_{2}\subseteq\kappa_{2}$ which consists of causal diamonds in $\{J\in \kappa_{2}\;|\;J\cap J'=\emptyset \text{ for any }J'\in \tilde{\kappa}_{1}\}$. Iterate this construction and we get a sequence $\{\tilde{\kappa}_{j}\}_{j=1}^{\infty}$ of subfamilies of $\kappa$ such that $\tilde{\kappa}_{j}$ satisfies
\[\left\{ \begin{aligned} 
  & (1) \;\textrm{$\tilde{\kappa}_{j}$ is pairwise disjoint}, \\
  & (2) \;J\in\tilde{\kappa}_{j}\Rightarrow J\cap J'=\emptyset \textrm{ for all $J'\in\tilde{\kappa}_{1}\cup\dotsb\cup\tilde{\kappa}_{j-1}$} \\
             \end{aligned}
             \right.\]
and is also maximal in $\kappa_{j}$ for all $j\in \mathbb{N}$. Then, we see that the family $\kappa'\coloneq\displaystyle\bigcup_{j=1}^{\infty}\tilde{\kappa}_{j}$ is pairwise disjoint. Moreover, take $J\in\kappa_{j}$ arbitrarily and we can get a $J'\in\tilde{\kappa}_{1}\cup\dotsb\cup\tilde{\kappa}_{j}$ with $J\cap J'\not=\emptyset$ by the maximality. Notice that $\mathrm{diam}_{d}(J)\le2\mathrm{diam}_{d}(J')$, and hence $J\subseteq \hat{J'}$ holds. Therefore, we have
\[E\subseteq\displaystyle\bigcup_{J\in\kappa}J\subseteq\displaystyle\bigcup_{J'\in\kappa'}\hat{J'}.\]

Next, assume moreover that any causal diamond in $X$ is closed and  $\kappa$ is a fine covering of $E$. Then, let $K\subseteq \kappa'$ be an arbitrarily taken finite subfamily. Take $x\in E\setminus \displaystyle\bigcup_{J\in K}J$. Since any causal diamond in $X$ is closed and $\kappa$ is a fine covering of $E$, we can take an open ball $B$ and $J'\in \kappa$ with
$x\in J' \subseteq B \subseteq E\setminus \displaystyle\bigcup_{J\in K}J$. Notice that we can get $J^{*}\in \kappa'$ such that $J'\cap J^{*}\not=\emptyset$ and $J'\subseteq\hat{J^{*}}$. It is obvious that such $J^{*}\in\kappa'$ is not an element of $K$. Therefore, we have $x\in \displaystyle\bigcup_{J'\in\kappa'\setminus K}\hat{J'}.$
\end{proof}

Next, with the preceding Lemma \ref{causal covering lemma}, we prove the following lemma.
\begin{Lemma}\label{volume upper bound of causal diamonds}
  Let $(X, d, \le, \ll, \tau)$ be a Lorentzian pre-length space satisfying the local causal enlargement property and $E\subseteq X$ with $\mathcal{V}^{s}_{d}(E)<\infty$. Assume that $\mathcal{V}^{s}_{d}$ is an outer regular measure, $\mathcal{V}^{s}_{d}(J(p,q))=0$ if $p$ and $q$ are null related, any causal diamond in $X$ is closed, and there is a compact set $A\subseteq X$ which satisfies $d(E, A^{c})>0$. Then for any $\epsilon>0$, we can get a $\mathcal{V}^{s}_{d}$-negligible set $N\subseteq E$ and then for any $x\in E\setminus N$ there exists $\delta_{x}>0$ such that
  \[\left\{ \begin{aligned}
& (1) \;J\in\mathcal{J} \\
& (2) \;x\in J\subseteq B_{d}(x, \delta_{x})\\
 \end{aligned} \right.
  \quad \Rightarrow
  \quad \mathcal{V}^{s}_{d}(E\cap J)\le (1+\epsilon)\rho_{s}(J).
\]

\end{Lemma}

\begin{proof}
  Let $\epsilon>0$ and $\delta>0$. Let $G$ be the set of points $x\in E$ such that for every $j\in \mathbb{N}$ we can find a causal diamond $J_{x}^{j}\in \mathcal{J}$ satisfying
  \[ \left\{ \begin{aligned}
& (1) \;x\in J_{x}^{j}\subseteq B_{d}(x, j^{-1})\\
& (2) \;\mathcal{V}^{s}_{d}(E\cap J_{x}^{j})>(1+\epsilon)\rho_{s}(J_{x}^{j}).\\
  \end{aligned} \right.\tag{B}\\\]
Assume that $\mathcal{V}^{s}_{d}(G)>0$. Since $\mathcal{V}^{s}_{d}$ is an outer regular measure and $\mathcal{V}^{s}_{d}(G)<\infty$, we can find an open set $O\subseteq X$ with $\mathcal{V}^{s}_{d}(O)<(1+\frac{\epsilon}{4})\mathcal{V}^{s}_{d}(G)$. We can assign $J_{x}^{j}$ which enjoys (B) for each $x\in G$ and $j\in\mathbb{N}$, and moreover since we assume $\mathcal{V}^{s}_{d}(J(p,q))=0$ if $p$ and $q$ are null related, we can let $J_{x}^{j}$ have timelike related vertices. Now notice that since $G\subseteq E \subseteq A$ and $d(G, A^{c})>0$, we can apply Lemma \ref{causal covering lemma}. Then, take constants $C>0$, $C_{1}\ge1$, and $C_{2}\ge2$ as in Lemma \ref{causal covering lemma}. Then, we set a family of causal diamonds $\Gamma_{\delta}$ as 
\[\Gamma_{\delta}\coloneq\{J_{x}^{j}\;|\;j\in\mathbb{N}, x\in G, \mathrm{diam}_{d}(J_{x}^{j})<\mathrm{min}\{C, \delta\}, J_{x}^{j}\subseteq O\}.\]
For any $x\in G$ with $j_{x}\in\mathbb{N}$ such that $B(x, j_{x}^{-1})\subseteq O$ and $j_{x}^{-1}<\frac{1}{2}\mathrm{min}\{C, \delta\}$, we see that for any $\lambda>0$ letting $j\in\mathbb{N}$ be $j^{-1}\le\min{(j_{x}^{-1}, \lambda)}$, $J_{x}^{j}\in\Gamma_{\delta}$ and $x\in J_{x}^{j}\subseteq B(x, j^{-1})\subseteq B(x, \lambda)$ follow. Thus, we have $\Gamma_{\delta}$ forms a fine cover of $G$ for any $\delta>0$. Then, applying Lemma $\ref{causal covering lemma}$ again, we can get a pairwise disjoint subfamily $\Gamma_{\delta,1}\subseteq\Gamma_{\delta}$ and finite subfamily $\Gamma_{\delta,2}\subseteq\Gamma_{\delta,1}$ such that 
\[\left\{ \begin{aligned}
& (1)\;G\subseteq \displaystyle\bigcup_{J\in\Gamma_{\delta,1}}\hat{J}\\
& (2)\displaystyle\sum_{J'\in\Gamma_{\delta,1}\setminus\Gamma_{\delta,2}}\mathcal{V}^{s}_{d}(J')<\frac{\epsilon}{4C_{2}^{s}}\mathcal{V}^{s}_{d}(G).\\
 \end{aligned} \right.\tag{C}\]
 Indeed, since we have 
\[\displaystyle\bigcup_{J\in\Gamma_{\delta,1}}J\subseteq O,\]
$\mathcal{V}^{s}_{d}(O)<\infty$, and $\Gamma_{\delta,1}$ is pairwise disjoint, at most countable elements of $\Gamma_{\delta,1}$ have positive volume, and thus we have
\[\mathcal{V}_{d}^{s}(\displaystyle\bigcup_{J\in\Gamma_{\delta,1}}J)=\displaystyle\sum_{J\in\Gamma_{\delta,1}}\mathcal{V}_{d}^{s}(J)<\infty.\]
Therefore, taking a sufficiently large finite subfamily, we can get $\Gamma_{\delta, 2}\subseteq\Gamma_{\delta, 1}$ satisfying (2) in (C). Moreover, since any causal diamond in $X$ is closed, we have
\[G\subseteq \displaystyle\bigcup_{J\in\Gamma_{\delta,2}}J\;\cup\displaystyle\bigcup_{J'\in\Gamma_{\delta,1}\setminus\Gamma_{\delta,2}}\hat{J'},\]
and thus the following holds:
\[\begin{aligned}
  \mathcal{V}^{s}_{d,C_{1}\cdot\delta}(G) 
  &\le \displaystyle\sum_{J\in\Gamma_{\delta,2}}\rho_{s}(J)+\displaystyle\sum_{J'\in\Gamma_{\delta,1}\setminus\Gamma_{\delta,2}}\rho_{s}(\hat{J'})\\
  &<\displaystyle\sum_{J\in\Gamma_{\delta,2}}\frac{1}{1+\epsilon}\mathcal{V}^{s}_{d}(J)+\displaystyle\sum_{J'\in\Gamma_{\delta,1}\setminus\Gamma_{\delta,2}}\frac{C_{2}^{s}}{1+\epsilon}\mathcal{V}^{s}_{d}(J')\\
  &<\frac{1}{1+\epsilon}\mathcal{V}^{s}_{d}(O)+\frac{C_{2}^{s}}{1+\epsilon}\cdot\frac{\epsilon}{4C_{2}^{s}}\mathcal{V}^{s}_{d}(G)\\
  &<\frac{1}{1+\epsilon}\cdot(1+\frac{\epsilon}{4})\mathcal{V}^{s}_{d}(G)+\frac{\epsilon}{4(1+\epsilon)}\mathcal{V}^{s}_{d}(G)\\
  &=\frac{1+\frac{\epsilon}{2}}{1+\epsilon}\mathcal{V}^{s}_{d}(G).
\end{aligned}\]
Then, letting $\delta\to 0$, we get $0<\mathcal{V}^{s}_{d}(G)<\mathcal{V}^{s}_{d}(G)<\infty$. This is a contradiction, and thus it follows that $G$ is $\mathcal{V}^{s}_{d}$-negligible. The claim holds.
\end{proof}

Before the proof of Theorem \ref{equality of integration}, we also show the following lemma.

\begin{Lemma}\label{estimate sets by Borel sets}
  Let $X$ be a Lorentzian pre-length space and $s\ge0$. Assume that any causal diamond in $X$ is Borel. Then, for any $A\subseteq X$, we find a Borel set $\tilde{A}$ such that $A \subseteq \tilde{A}$ and $\mathcal{V}^{s}(A)=\mathcal{V}^{s}(\tilde{A})$.
\end{Lemma}
\begin{proof}
  Let $A\subseteq X$. Assume $M\coloneq\mathcal{V}^{s}(A)<\infty$. If not $\tilde{A}=X$ is the desired Borel set. Since $M<\infty$, for each $j\in\mathbb{N}$, $\mathcal{V}^{s}_{j^{-1}}(A)<\infty$ holds, and thus we can find a causal $j^{-1}$-covering, $L_{j}$ of $A$ such that 
  \[\displaystyle\sum_{J\in L_{j}}\rho_{s}(J)< M+j^{-1}.\]
  Let $L\coloneq\displaystyle\bigcap_{j=1}^{\infty}\displaystyle\bigcup_{J\in L_{j}}J$, and notice that since any causal diamond in $X$ is Borel, $L$ is Borel. We can obviously see $A\subseteq L$ and that for any $j\in\mathbb{N}$, $\mathcal{V}^{s}_{j^{-1}}(L)<M+j^{-1}$ holds. Thus, by $j\to\infty$, $\mathcal{V}^{s}(L)\le M$ follows, and it implies $\mathcal{V}^{s}(L)=\mathcal{V}^{s}(A)$.
\end{proof}

With Lemma \ref{causal covering lemma}, Lemma \ref{volume upper bound of causal diamonds}, and Lemma \ref{estimate sets by Borel sets}, we show Theorem \ref{equality of integration}.

\begin{proof}[Proof of Theorem \ref{equality of integration}]
  First, we prove the inequality
  \[\int_{X}^{\bullet}fd\mathcal{V}^{s}_{d}\le\int_{X}^{\ast}fd\mathcal{V}^{s}_{d}.\]
We assume that the right hand side of the above inequality is finite. Let $\epsilon>0$. Then, by Lemma \ref{estimation by step functions} we can take a step function of $\{(a_{i}, A_{i})\}_{i\in\mathbb{N}}$ such that
\[\left\{\begin{aligned}
  &(1) \;\textrm{$a_{i}>0$ and $A_{i}$ is $\mathcal{V}^{s}_{d}$-measurable},\\
  &(2) \;f\le\displaystyle\sum_{i=1}^{\infty}a_{i}\chi_{A_{i}},\\
  &(3) \;\displaystyle\sum_{i=1}^{\infty}a_{i}\mathcal{V}^{s}_{d,\delta}(A_{i})\le\displaystyle\sum_{i=1}^{\infty}a_{i}\mathcal{V}^{s}_{d}(A_{i})\le\int_{X}^{\ast}fd\mathcal{V}^{s}_{d}+\frac{\epsilon}{2}.
\end{aligned}\right.  
\]
Since $\mathcal{V}^{s}_{d, \delta}(A_{i})<\infty$, for each $i\in \mathbb{N}$ we can take a causal $\delta$-covering $\{A^{(i)}_{j}\}_{j\in\mathbb{N}}$ of $A_{i}$ such that
\[\displaystyle\sum_{j=1}^{\infty}\rho_{s}(A^{(i)}_{j})<\mathcal{V}^{s}_{d,\delta}(A_{i})+\frac{\epsilon}{2^{i+1}a_{i}}.\]
Notice that we have
\[f\le\displaystyle\sum_{i,j=1}^{\infty}a_{i}\chi_{A^{(i)}_{j}},\]
and thus $\{(a_{i}, A^{(i)}_{j})\}_{i,j\in\mathbb{N}}$ is a weighted causal $\delta$-covering of $f$. Then, from the definition of the weighted $\delta$-integral of $f$, it follows that 
\[\begin{aligned}
\int_{X}^{\bullet}fd\mathcal{V}^{s}_{d,\delta}&\le\displaystyle\sum_{i,j=1}^{\infty}a_{i}\rho_{s}(A^{(i)}_{j})\\
                                            &=\displaystyle\sum_{i=1}^{\infty}a_{i}\big{(}\displaystyle\sum_{j=1}^{\infty}\rho_{s}(A^{(i)}_{j})\big{)}\\
                                            &<\displaystyle\sum_{i=1}^{\infty}a_{i}\mathcal{V}^{s}_{d,\delta}(A_{i})+\frac{\epsilon}{2}\\
                                            &<\int_{X}^{\ast}fd\mathcal{V}^{s}_{d}+\epsilon.
\end{aligned}\]
Then, letting $\epsilon\to0$ first and $\delta\to0$ next, we get the desired inequality. 

Next, we prove the reverse inequality, 
  \[\int_{X}^{\ast}fd\mathcal{V}^{s}_{d}\le\int_{X}^{\bullet}fd\mathcal{V}^{s}_{d}.\]
 We assume that the right hand side is finite. Let $A\coloneq\{x\in X\;|\;f(x)>0\}$. Notice that by Lemma \ref{estimate sets by Borel sets} we can assume $A$ is a Borel set. As the first step, we assume that $A$ satisfies the conditions mentioned in Lemma $\ref{volume upper bound of causal diamonds}$, namely
 \[\left\{\begin{aligned}
  &(1) \;\mathcal{V}^{s}_{d}(A)<\infty, \\
  &(2) \;\textrm{There exists a compact set $B\subseteq X$ satisfying $d(A, B^{c})>0$}. \\
\end{aligned}\right.\]
Let $\epsilon>0$. For $j\in\mathbb{N}$, let $W_{j}$ be the set of $x\in A$ which satisfies the property;
\[\left\{ \begin{aligned}
&J\in\mathcal{J} \\
&x\in J\subseteq B_{d}(x, j^{-1})\\
 \end{aligned} \right.
  \quad \Rightarrow
  \quad \mathcal{V}^{s}_{d}(A\cap J)\le (1+\epsilon)\rho_{s}(J).
\]
Then, since we can apply Lemma $\ref{volume upper bound of causal diamonds}$ for $A$, we have
\[A=N\cup\displaystyle\bigcup_{j\in\mathbb{N}}W_{j},\]
where $N\subseteq A$ is a $\mathcal{V}^{s}_{d}$-negligible set. 
On the other hand, since the causal weighted integration of $f$ regarding to $\mathcal{V}_{d}^{s}$ is finite, for each $j\in\mathbb{N}$ we can take a weighted causal $j^{-1}$-covering $\{(a_{i}, A^{(j)}_{i})\}_{i\in\mathbb{N}}$ of $f$ such that 
\[\displaystyle\sum_{i\in\mathbb{N}}a_{i}\rho_{s}(A^{(j)}_{i})\le\int_{X}^{\bullet}fd\mathcal{V}^{s}_{d,j^{-1}}+\epsilon.\]
Let $I^{j}\coloneq\{i\in \mathbb{N}\;|\;W_{j}\cap A^{(j)}_{i}\not=\emptyset\}$ for each $j\in\mathbb{N}$, and then we can see that
\[f\cdot\chi_{W_{j}}\le\displaystyle\sum_{i\in I^{j}}a_{i}\chi_{A\cap A^{(j)}_{i}}.\]
Then, by integrating both sides, we have
\[\begin{aligned}
\int_{X}^{\ast}f\cdot\chi_{W_{j}}d\mathcal{V}^{s}_{d}&\le\int_{X}^{\ast}\displaystyle\sum_{i\in I^{j}}a_{i}\chi_{A\cap A^{(i)}_{j}}d\mathcal{V}^{s}_{d}\\
                                            &=\displaystyle\sum_{i\in I^{j}}a_{i}\mathcal{V}^{s}_{d}(A\cap A^{(j)}_{i})\\
                                            &<\displaystyle\sum_{i\in I^{j}}a_{i}(1+\epsilon)\rho_{s}(A^{(j)}_{i})\\
                                            &\le(1+\epsilon)\big{(}\int_{X}^{\bullet}fd\mathcal{V}^{s}_{d,j^{-1}}+\epsilon\big{)}.
\end{aligned}\]
The equality in the second line holds since $A$ is Borel. Then, letting $j\to\infty$ first and $\epsilon\to 0$ next, we have the desired inequality thanks to Lemma \ref{flipping limit}.

For general $f$, fix a point $x\in X$ and a sequence $\{X_{i}\}_{i=1}^{\infty}$ of subsets of $X$ satisfying $X_{i}\uparrow X$ and $\mathcal{V}^{s}_{d}(X_{i})<\infty$. The existence of such sequence is guaranteed since $\mathcal{V}_{d}^{s}$ is $\sigma$-finite. Then, notice $B_{d}(x, i)$ is compact for any $i\in\mathbb{N}$ since $d$ is proper, and we see $g_{i}\coloneq f\cdot \chi_{X_{i}\cap B_{d}(x, i)}$ satisfies the above assumption. Therefore, we have
\[\int_{X}^{\ast}g_{i}d\mathcal{V}^{s}_{d}\le\int_{X}^{\bullet}g_{i}d\mathcal{V}^{s}_{d}\le\int_{X}^{\bullet}fd\mathcal{V}^{s}_{d},\]
and since $g_{i}\to f$ as $i\to \infty$,
\[\int_{X}^{\ast}fd\mathcal{V}^{s}_{d}\le\int_{X}^{\bullet}fd\mathcal{V}^{s}_{d}\]
follows again by Lemma \ref{flipping limit}.
\end{proof}

\begin{Remark}
  One may think that the causal weighted integration of a function can be defined not necessary with step functions of causal diamonds. Indeed, if we consider a family $\mathcal{C}$ of sets in a Lorentzian pre-length space $X$ defined as
  \[\mathcal{C}\coloneq\left\{j\subseteq X\;|\textrm{ $j$ is included in a causal diamond $J(p,q)$ and $p,q\in j$}\;\right\}\]
and define $\rho_{s}(j)$ as $\rho_{s}(J)$ with any causal diamond $J$ including $j$, $\rho_{s}(j)$ is well-defined. If we define the causal weighted integration by step functions of sets in $\mathcal{C}$ rather than causal diamonds, we can replace the assumption that the map $U$ is uniformly $d$-controlling in Theorem \ref{Lcoarea inequality for causal weighted integration} with that $U$ is causality preserving and uniformly continuous. Nevertheless, we have to define the causal weighted integral with step functions of causal diamonds since Lemma \ref{volume upper bound of causal diamonds} is valid only for causal diamonds.
\end{Remark}

Theorem \ref{LEilenberg's inequality} is a direct implication of Theorem \ref{Lcoarea inequality for causal weighted integration} and Theorem \ref{equality of integration}.

\begin{Remark}
  In \cite[Proposition 8.1.12]{DoC}, we see that any locally finite Borel measure on a Polish space is regular i.e. it is outer and inner regular.
\end{Remark}

\begin{Remark}
  In $\cite{Mc;Sa}$, McCann and S\"{a}mann discussed the enlargement of causal diamonds through the notion of cylindrical neighborhoods, and we can see that any strongly causal and causally plain $C^{0}$-spacetime admits the local causal enlargement property. Moreover, they consider Borel measure with doubling properties and a covering lemma which our covering lemma is similar to.
\end{Remark}

\section{Measurability of the integrand}
In this section, we prove the measurability of the integrand in Theorem \ref{LEilenberg's inequality} and conclude with Theorem \ref{LORENTZIAN COAREA INEQUALITY}.

\begin{Theorem}[Measurability of the integrand]\label{measurability of the integrand}
  Let $(X, d_{X}, \le_{X}, \ll_{X}, \tau_{X})$ and $(Y, d_{Y}, \le_{Y}, \ll_{Y}, \tau_{Y})$ be Lorentzian pre-length spaces, and $F:X\to Y$ be a continuous map. Assume all the assumptions in Theorem \ref{LEilenberg's inequality} and in addition that $\mathcal{V}^{s}_{d_{X}}$ is an inner regular measure on $X$, $d_{X}$ is proper, and $X$ satisfies the causal estimation property. Then, for any $E\subseteq X$ such that is $\mathcal{V}^{s}_{d_{X}}$-measurable and satisfies $\mathcal{V}^{s}_{d_{X}}(E)<\infty$, the map $y\mapsto\mathcal{V}^{s-t}_{d_{X}}(F^{-1}(y)\cap E)$ is a $\mathcal{V}^{t}_{d_{Y}}$-measurable function.
\end{Theorem}
$\linebreak$To this end we introduce a variant of the Lorentzian Hausdorff measure.

\begin{Definition}[Strong Lorentzian Hausdorff measure]
  Let $(X, d, \le, \ll, \tau)$ be a Lorentzian pre-length space. Let $E\subseteq X$ and $\delta>0$. We call a subfamily $\{I_{i}\}_{i=1}^{\infty}$ of $\mathcal{I}\coloneq\{I(p,q)\;|\;p\ll q\}\cup\{\emptyset\}$ a \textit{chronological $\delta$-covering} of $E$ if it holds
  \[\left\{ \begin{aligned}
& \mathrm{diam}(I_{i})\le\delta,\\
& E\subseteq\displaystyle\bigcup_{i=1}^{\infty}I_{i}.\\
 \end{aligned} \right.\]
Then, for $s\ge0$ we define
  \[\mathcal{M}^{s}_{\delta}(E)\coloneq\inf \left\{\displaystyle\sum_{i=1}^{\infty} \rho_{s} (I_{i})\;\Bigg{|}\;\{I_{i}\}_{i=1}^{\infty} \textrm{is a chronological $\delta$-covering of $E$}\right\},\]
  and
  \[\mathcal{M}^{s}(E)\coloneq\displaystyle\lim_{\delta\to0}\mathcal{M}^{s}_{\delta}(E),\]
where for $p\ll q$ we set $\rho_{s}(I(p,q))=\omega_{s}\tau(p,q)^{s}$, and we let $\rho_{s}(\emptyset)=0$ and $\rho_{s}(I(p,q))=\infty$ for $I(p,q)\in \mathcal{I}$ with $\tau(p,q)=\infty$. 
\end{Definition}

Through the causal estimation property, we have the coincidence of Lorentzian and strong Lorentzian Hausdorff measures over any subset which is included in a compact set.
 
\begin{Lemma}\label{Estimation of causal diamond by chronological dismonds}
Let $(X, d, \le, \ll, \tau)$ be a Lorentzian pre-length space, $s\ge0$, and $E\subseteq X$. Assume that $X$ satisfies the causal estimation property and there exists a compact set $A\subseteq X$ such that $E\subseteq A$ and $d(E, A^{c})>0$. Then, we can find a constant $C>0$ such that $\mathcal{M}^{s}_{\delta}(G)=\mathcal{V}^{s}_{\delta}(G)$ for any $\delta<C$ and $G\subseteq E$. 
\end{Lemma}
\begin{proof}
Since $X$ satisfies the causal estimation property, we can assign causal estimation neighborhood for each $x\in A$. Thus, since $A$ is compact and $D\coloneq d(E, A^{c})>0$, letting $C\in (0, D)$ be a Lebesgue number for the covering $\{U_{x}\}_{x\in A}$ of $A$, any chronological diamond $I(p,q)$ (resp. causal diamond $J(p,q)$) with $\textrm{diam}_{d}(I(p,q))<C$ and $I(p,q)\cap E\not=\emptyset$ (resp. with $\textrm{diam}_{d}(J(p,q))<C$ and $J(p,q)\cap E\not=\emptyset$) is included in a causal estimation neighborhood.

Let $G\subseteq E$. Take $\delta\in(0, C)$. For any chronological $\delta$-covering $\{I(p_{i}, q_{i})\}_{i=1}^{\infty}$ of $G$ with $G\cap I(p_{i},q_{i})\not=\emptyset$ for any $i\in\mathbb{N}$, we see that $\{J(p_{i}, q_{i})\}_{i=1}^{\infty}$ forms a causal $\delta$-covering of $G$, and thus it follows that
\[\mathcal{V}^{s}_{\delta}(G)\le\displaystyle\sum_{i=1}^{\infty}\rho_{s}(J(p_{i}, q_{i}))=\displaystyle\sum_{i=1}^{\infty}\rho_{s}(I(p_{i}, q_{i})).\]
Taking the infimum over such chronological $\delta$-coverings of $G$, we have $\mathcal{V}^{s}_{\delta}(G)\le\mathcal{M}^{s}_{\delta}(G)$. Next, we show the reverse inequality. Let $\epsilon> 1$ and take any causal $\delta$-covering $\{J(p_{i},q_{i})\}_{i=1}^{\infty}$ of $G$ with $G\cap J(p_{i},q_{i})\not=\emptyset$ for any $i\in\mathbb{N}$. We assign a chronological ($\epsilon$, $2^{-i}(\epsilon-1)$)-estimation $I(\tilde{p}_{i},\tilde{q}_{i})$ for each $J(p_{i},q_{i})$. Then, we have
 \[\mathcal{M}^{s}_{\epsilon\cdot\delta}(G)\le\displaystyle\sum_{i=1}^{\infty}\rho_{s}(I(\tilde{p}_{i}, \tilde{q}_{i}))\le(\epsilon-1)+\displaystyle\sum_{i=1}^{\infty}\rho_{s}(J(p_{i}, q_{i})),\]
 and taking the infimum over all causal $\delta$-coverings of $G$ it follows
 \[\mathcal{M}^{s}_{\epsilon\cdot\delta}(G)\le(\epsilon-1)+\mathcal{V}^{s}_{\delta}(G).\]
 Letting $\epsilon\to1$, $\mathcal{M}^{s}_{\delta}(G)\le\mathcal{V}^{s}_{\delta}(G)$ follows.
\end{proof}

\begin{Remark}
When we merely require the coincidence of $\mathcal{M}^{s}$ and $\mathcal{V}^{s}$, we can weaken conditions in Lemma \ref{Estimation of causal diamond by chronological dismonds}. First, we can replace the condition, $\textrm{diam}_{d}(I(p,q))=\textrm{diam}_{d}(J(p,q))$ with that there exists a constant $C\ge1$ such that $\textrm{diam}_{d}J(p,q)\le C\textrm{diam}_{d}(I(p,q))$. Moreover, the requirement for  any causal $\epsilon$-estimation $I(\tilde{p}, \tilde{q})$ of $J(p,q)$ to hold $\textrm{diam}_{d}(I(\tilde{p}, \tilde{q}))\le\epsilon\cdot\textrm{diam}_{d}(J(p,q))$ is weakened to that $I(\tilde{p}, \tilde{q})\le C'\cdot\textrm{diam}_{d}(J(p,q))$ holds for a constant $C'\ge1$ which does not depend on the choice of $J(p,q)$. 
\end{Remark}

Then, we prove Theorem \ref{measurability of the integrand}.

\begin{proof}[Proof of Theorem \ref{measurability of the integrand}]
 Since $\mathcal{V}^{s}_{d_{X}}$ is inner regular, we have a decomposition of $E$ as follows:
  \[E=N\cup\displaystyle\bigcup_{i=1}^{\infty}K_{i},\]
where $N$ is $\mathcal{V}^{s}_{d_{X}}$-negligible, $K_{i}$ is compact, and $K_{i}\subseteq K_{i+1}$ holds for each $i\in \mathbb{N}$.
Notice that by Theorem $\ref{LEilenberg's inequality}$ and Lemma \ref{zero integration imply zero a.e.}, for $\mathcal{V}^{t}_{d_{Y}}$-almost every $y\in Y$, ${V}^{s-t}_{d_{X}}(F^{-1}(y)\cap N)=0$ holds. Then, we have
\[\mathcal{V}^{s-t}_{d_{X}}(F^{-1}(y)\cap E)=\mathcal{V}^{s-t}_{d_{X}}\left(F^{-1}(y)\cap \displaystyle\bigcup_{i=1}^{\infty}K_{i} \right)=\displaystyle\lim_{i\to \infty}\mathcal{V}^{s-t}_{d_{X}}(F^{-1}(y)\cap K_{i}).\]
Thus, it remains to show that $y\mapsto\mathcal{V}^{s-t}_{d_{X}}(F^{-1}(y)\cap K)$ is a $\mathcal{V}^{t}_{d_{Y}}$-measurable function for any compact subset $K\subseteq X$. To this end, it suffices to prove that for any compact set $K\subseteq X$ and any $h\in\mathbb{R}$, $Y_{h}(K)\coloneq \{y\in Y\;|\;\mathcal{V}^{s-t}_{d_{X}}(F^{-1}(y)\cap K)\le h\}$ is a $\mathcal{V}^{t}_{d_{Y}}$-measurable set. Since $Y_{h}(K)=\emptyset$ if $h<0$, we can assume $h\ge 0$. Moreover, we observe 
\[Y_{h}(K)=\displaystyle\bigcap_{j=1}^{\infty}\left\{y\in Y\;\Bigg{\vert}\; \mathcal{V}^{s-t}_{d_{X},\frac{1}{j}}(F^{-1}(y)\cap K)<h+\frac{1}{j}\right\},\]
and thus for the claim it suffices to show that $Z_{\alpha, \beta}(K)\coloneq\big{\{}y\in Y\;|\; \mathcal{V}^{s-t}_{d_{X},\alpha}(F^{-1}(y)\cap K)<\beta\big{\}}$ is a $\mathcal{V}^{t}_{d_{Y}}$-measurable set for sufficiently small $\alpha>0$ and any $\beta>0$. Actually, we can show that $Z_{\alpha, \beta}(K)$ is an open set in the following way. Since we assume $X$ satisfies the causal estimation property, $d_{X}$ is proper, and $\alpha>0$ is sufficiently small, applying Lemma \ref{Estimation of causal diamond by chronological dismonds} for $K$, we have $\mathcal{M}^{s-t}_{d_{X},\alpha}=\mathcal{V}^{s-t}_{d_{X},\alpha}$ over $K$ and especially $Z_{\alpha, \beta}(K)=\big{\{}y\in Y\;|\; \mathcal{M}^{s-t}_{d_{X},\alpha}(F^{-1}(y)\cap K)<\beta\big{\}}$. Thus, take any $y\in Z_{\alpha, \beta}(K)$, and we can get a familly of chronological diamonds $\{I_{i}\}_{i=1}^{\infty}$ with $\mathrm{diam}_{d_{X}}(I_{i})<\alpha$ and $I_{i}\cap (F^{-1}(y)\cap K)\not=\emptyset$ for all $i\in\mathbb{N}$ such that
\[\left\{ \begin{aligned}
&F^{-1}(y)\cap K \subseteq \displaystyle\bigcup_{i=1}^{\infty}I_{i},\\
&\displaystyle\sum_{i=1}^{\infty}\rho_{s-t}(I_{i})<\beta.\\
 \end{aligned} \right.\]
Then, it is sufficient for the claim that for any sequence $\{y_{n}\}_{n=1}^{\infty}$ in $Y$ converging to $y$, we can find $M\in \mathbb{N}$ such that for all $m\ge M$, $F^{-1}(y_{m})\cap K \subseteq \displaystyle\bigcup_{i=1}^{\infty}I_{i}$ holds. Assume that there exists a sequence $\{y_{n}\}_{i=1}^{\infty}$ converging to $y$ but we can find $M\in\mathbb{N}$ mentioned above. Then, we can get a subsequence $\{y_{nk}\}_{k=1}^{\infty}\subseteq \{y_{n}\}_{n=1}^{\infty}$ such that $\Big{(}F^{-1}(y_{nk})\cap K\Big{)} \setminus \displaystyle\bigcup_{i=1}^{\infty}I_{i}\not=\emptyset$ and a sequence $\{x_{nk}\}_{k=1}^{\infty}$ in $X$ such that $x_{nk}\in \Big{(}F^{-1}(y_{nk})\cap K\Big{)} \setminus \displaystyle\bigcup_{i=1}^{\infty}I_{i}$. Notice that since $\{x_{nk}\}_{k=1}^{\infty}\subseteq K$ and $I_{i}$ is open, we can get a non-relabeled subsequence of $\{x_{nk}\}_{k=1}^{\infty}$ which converges to a point $x\in K\setminus\displaystyle\bigcup_{i=1}^{\infty}I_{i}$. Then, since $F$ is continuous, we see that $F(x)=\displaystyle\lim_{k\to\infty}F(x_{nk})=\displaystyle\lim_{k\to \infty}y_{nk}=y$, and it implies $x\in F^{-1}(y)\cap K\subseteq \displaystyle\bigcup_{i=1}^{\infty}I_{i}$. This is a contradiction, and the desired claim is attained.
\end{proof}

To conclude, Theorem \ref{LORENTZIAN COAREA INEQUALITY} is implied by Theorem \ref{LEilenberg's inequality} and Theorem \ref{measurability of the integrand} immediately.

\section{Outlook}
When the inequality in Theorem \ref{LORENTZIAN COAREA INEQUALITY} becomes equality is one natural question. In the case of metric spaces mentioned in Theorem \ref{coarea inequality}, we can see that equality holds when the domain, $X$ and the target space, $Y$ are rectifiable metric spaces. It is proved through the differentiability of (local) Lipschitz maps which follows from Rademacher's theorem. Therefore, to get the equality condition of Lorentzian coarea inequality, one possible direction is to get a Lorentzian counterpart of Rademacher's theorem for some class of maps including uniformly $d$-controlling and timelike Lipschitz maps.

The coarea inequality or coarea formula is widely used in the geometry of metric spaces. One important application of that is to prove the equivalence of $(q,1)$-Poincar\'{e} inequality with $1\le q<\infty$ and relative isoperimetric inequality (see $\cite{R}$). Through the Lorentzian coarea inequality, we can consider problems known in the geometry of metric spaces.

\end{document}